\documentclass[12pt]{amsart}

\usepackage{verbatim}
\usepackage{fullpage}
\usepackage{amsfonts}
\usepackage{bbm}
\usepackage{url}

\usepackage{color}
\usepackage{colortbl}
\usepackage{graphicx}
\usepackage{epstopdf}
\usepackage{amsmath, amsthm, amssymb}
\usepackage{fancybox}
\newcommand\iprec{\mathrel{\ooalign{$\prec$\cr
 \,\raise0.85ex\hbox{\scriptsize$\circ$}\cr}}}

\newcounter{Theorem}

\numberwithin{equation}{section}
\numberwithin{Theorem}{section}

\theoremstyle{plain} 
\newtheorem{thm}[Theorem]{Theorem}
\newtheorem{cor}[Theorem]{Corollary}
\newtheorem{lem}[Theorem]{Lemma}
\newtheorem{prop}[Theorem]{Proposition}
\usepackage{upgreek}

\theoremstyle{definition}

\theoremstyle{remark}

\begin{document}

\title{Tremain equiangular tight frames}

\begin{abstract}
Equiangular tight frames provide optimal packings of lines through the origin.
We combine Steiner triple systems with Hadamard matrices to produce a new infinite family of equiangular tight frames.
This in turn leads to new constructions of strongly regular graphs and distance-regular antipodal covers of the complete graph.
\end{abstract}

\author{Matthew Fickus}
\address{Department of Mathematics and Statistics, Air Force Institute of Technology, Wright-Patterson Air Force Base, OH 45433, USA}
\email{Matthew.Fickus@gmail.com}
\author{John Jasper}
\address{Department of Mathematical Sciences, University of Cincinnati, Cincinnati, OH 45221, USA}
\email{jasperjh@ucmail.uc.edu}
\author{Dustin G.\ Mixon}
\address{Department of Mathematics and Statistics, Air Force Institute of Technology, Wright-Patterson Air Force Base, OH 45433, USA}
\email{dustin.mixon@gmail.com}
\author{Jesse Peterson}
\address{European Office of Aerospace Research and Development, Air Force Office of Scientific Research, London, UK}
\email{Jesse.Peterson.1@us.af.mil}


\thanks{This research was supported by an AFOSR Young Investigator Research Program award, NSF Grant No.\ DMS-1321779, and AFOSR Grant No.\ F4FGA05076J002. 
The views expressed in this article are those of the authors and do not reflect the official policy or position of the United States Air Force, Department of Defense, or the U.S.\ Government.}

\date{\today}

\maketitle

\section{Introduction}

Consider the following problem:
Given a finite-dimensional Hilbert space $\mathcal{H}$ and some positive integer $N$, find vectors $\varphi_1,\ldots,\varphi_N\in \mathcal{H}$ of unit length that minimize \textbf{coherence}:
\[
\kappa:=\max_{\substack{i,j\in\{1,\ldots,N\}\\i\neq j}}|\langle\varphi_i,\varphi_j\rangle|.
\]
Geometrically, this problem amounts to packing lines through the origin.
Given their resemblance to error-correcting codes, it comes as no surprise that ensembles of minimal coherence, called \textbf{Grassmannian frames}, lead to communication protocols with minimal cross-talk.
Since their introduction by Strohmer and Heath in 2003~\cite{StrohmerH:03}, Grassmannian frames have received considerable attention in the finite frame theory community~\cite{FickusM:15}.

There were several precursors to the modern study of Grassmannian frames.
One notable example is a 1974 paper of Welch~\cite{Welch:74}, which provides various lower bounds for coherence in terms of the dimension $M$ of $\mathcal{H}$.
This first of these bounds is
\begin{equation}
\label{eq.welch bound}
\kappa
\geq\sqrt{\frac{N-M}{M(N-1)}},
\end{equation}
cf.\ Rankin's 1956 paper~\cite{Rankin:56}.
This has since been dubbed the \textbf{Welch bound}, and one can show that the coherence of an ensemble of unit-length vectors meets equality in the Welch bound precisely when the ensemble forms an \textbf{equiangular tight frame}~\cite{StrohmerH:03}, meaning $|\langle \varphi_i,\varphi_j\rangle|$ is some fixed constant for every choice of $i$ and $j$ with $i\neq j$ (i.e., the vectors are ``equiangular'') and furthermore, the operator $\mathbf{S}\colon\mathcal{H}\rightarrow\mathcal{H}$ defined by
\[
\mathbf{S}x
:=\sum_{i=1}^N \langle x,\varphi_i\rangle \varphi_i
\qquad \forall x\in\mathcal{H}
\]
is some multiple of the identity operator (i.e., the ensemble forms a so-called ``tight frame'').
By exhibiting equality in the Welch bound, equiangular tight frames (ETFs) are necessarily Grassmannian.

The significance of equiangularity brings us to another precursor of sorts.
In 1966, van~Lint and Seidel~\cite{vanLintS:66} first introduced a useful identification between real equiangular ensembles and graphs, in which vertices correspond to vectors, and edges are drawn according to the sign of the corresponding inner product.
Since the Gram matrix $\mathbf{G}=[\langle \varphi_i,\varphi_j\rangle]_{i,j=1}^N$ can be expressed in terms of the adjacency matrix $\mathbf{A}$ of this graph, one may then identify spectral properties of $\mathbf{A}$ with those of $\mathbf{G}$.
This is particularly important in the case of real ETFs, as the tightness criterion implies that $\mathbf{G}^2$ is a multiple of $\mathbf{G}$, which in turn forces $\mathbf{A}$ to satisfy a related quadratic equation.
Indeed, more recent treatments of this identification~\cite{Waldron:09,FickusW:15} have established a one-to-one correspondence between real ETFs and a family of \textbf{strongly regular graphs} (SRGs), namely, graphs with the property that every vertex has the same number $k$ of neighbors, that adjacent vertices have the same number $\lambda$ of common neighbors, and that non-adjacent vertices also have the same number $\mu$ of common neighbors.
Given the maturity of the SRG literature, this identification has compelled finite frame theorists to direct their attention toward complex ETFs, where open problems abound.

However, recent developments in ETF research leave one questioning the insurmountable maturity of the SRG literature.
In particular, a decades-old SRG construction by Goethals and Seidel~\cite{GoethalsS:70} (which leverages a particular family of real equiangular ensembles) was recently generalized in~\cite{FickusJMPW:15} to construct new SRGs from known real ETFs exhibiting certain symmetries.
The notion that ETF research could inform the design of graphs was further suggested by Coutinho, Godsil, Shirazi and Zhan~\cite{CountinhoGSZ:15}, who established that complex ETFs with Gram matrices consisting of prime roots of unity can be used to construct distance-regular antipodal covers of the complete graph (referred to as \textbf{DRACKNs}).

The present paper produces a new infinite family of complex ETFs that the authors discovered by generalizing Example~7.10 in Tremain's notes~\cite{Tremain:09}; we call these \textbf{Tremain ETFs}.
The construction can be written as a real ETF infinitely often, and in some cases, the corresponding SRGs are new (this is established with the help of Brouwer's table of strongly regular graphs~\cite{Brouwer:online,Brouwer:96}).
Indeed, the real ETFs in this paper are the first to be discovered in the last two decades.
In addition, some complex cases of the construction lead to a new infinite family of DRACKNs.

In the next section, we provide some preliminary information about ETFs.
Section~3 then gives our main result, namely, the construction of Tremain ETFs.
We conclude in Section~4 with consequences of this construction to combinatorial design.

\section{Preliminaries}

Throughout, we use $\{\delta_s\}_{s\in\mathcal{S}}$ to denote any fixed orthonormal basis of $\mathbb{C}^\mathcal{S}$.
For any $x\in\mathbb{C}^\mathcal{S}$, we may refer to the $s$th entry of $x$ with respect to this basis $x(s):=\langle x,\delta_s\rangle$.
Also, we identify a given vector $x\in\mathcal{H}$ with the corresponding linear operator $x\colon\mathbb{C}\rightarrow\mathcal{H}$ defined by $x(a):=ax$ for every $a\in\mathbb{C}$.
In particular, this identification allows us to express the adjoint operator $x^*\colon\mathcal{H}\rightarrow\mathbb{C}$ given by $x^*:=\langle\cdot,x\rangle$, along with the outer product operator $xx^*$, etc.
Finally, for any positive integer $k$, we denote $[k]:=\{1,\ldots,k\}$.

In this paper, it is convenient to scale the vectors in an ETF so that the inner products between any two vectors lies in the complex unit circle $\mathbb{T}$.
This scaling then forces the squared norm of each vector to be the reciprocal of the Welch bound \eqref{eq.welch bound}.
For example, consider any $N\times N$ complex Hadamard matrix, that is, an $N\times N$ matrix $\mathbf{H}$ with entries in $\mathbb{T}$ such that $\mathbf{H}\mathbf{H}^*=N\mathbf{I}_N$.
Then removing any row from this matrix will produce an $(N-1)\times N$ matrix $\mathbf{\Phi}$ whose columns form a so-called \textbf{unimodular simplex}.
The fact that such vectors form an ETF is easy to check:
In this case, $M=N-1$, and so the Welch bound \eqref{eq.welch bound} is $1/(N-1)$, and indeed, each column has squared norm $N-1$, whereas inner products between distinct columns lie in $\mathbb{T}$.
If we let $a$ denote the row of $\mathbf{H}$ that was removed to produce $\mathbf{\Phi}$, then we say its entries $\{a_i\}_{i=1}^N$ form the \textbf{Naimark complement} of the columns $\{\varphi_i\}_{i=1}^N$ of $\mathbf{\Phi}$.
The Naimark complement is useful because sums of inner products satisfy
\begin{equation}
\label{eq.naimark}
\langle \varphi_i,\varphi_j\rangle+a_i\overline{a_j}
=\left\{\begin{array}{cl}N&\text{if }i=j\\0&\text{else}\end{array}\right.
\end{equation}
by identifying the sum as an inner product between corresponding columns of $\mathbf{H}$.

Another useful family of ETFs is the so-called \textbf{Steiner ETFs}, first constructed in~\cite{FickusMT:12} from Steiner systems.
A \textbf{$(2,K,V)$-Steiner system} is a collection $\mathcal{V}$ of $V$ points together with a collection $\mathcal{B}$ of size-$K$ subsets of $\mathcal{V}$ (called \textbf{blocks}) with the property that each $2$-element subset of $\mathcal{V}$ is contained in exactly one block.
We write $B=|\mathcal{B}|$.
Also, it turns out that each point is contained in the same number $R$ of blocks.
These parameters satisfy the following identities:
\begin{equation}
\label{eq.steiner identities}
VR=BK,
\qquad
V-1=R(K-1).
\end{equation}
In this paper, we introduce a certain family of operators that will simplify the construction of Steiner ETFs, and help in the construction of Tremain ETFs in the next section.
In particular, for each $v\in\mathcal{V}$, consider an \textbf{embedding operator} $\mathbf{E}_v\colon\mathbb{C}^R\rightarrow\mathbb{C}^\mathcal{B}$ such that for every $b\in\mathcal{B}$ with $v\in b$, there exists $r\in[R]$ such that $\mathbf{E}_v\delta_r=\delta_b$.
For example, suppose $\mathcal{B}=\{b_1,\ldots,b_7\}$ and that $v$ is a member of $R=3$ blocks, specifically, $b_1$, $b_4$ and $b_5$.
Then we may take $\mathbf{E}_v$ to be any linear operator that sends $\{\delta_1,\delta_2,\delta_3\}$ to $\{\delta_{b_1},\delta_{b_4},\delta_{b_5}\}$.
In general, since $\mathbf{E}_v$ maps the orthonormal basis $\{\delta_r\}_{r=1}^R$ onto an orthonormal sequence $\{\delta_b\}_{b\in\mathcal{B},~\! v\in b}$ of the same cardinality, $\mathbf{E}_v$ is an isometry.
These embedding operators satisfy a few useful properties:

\begin{lem}
\label{lem.upsampling}
Given a $(2,K,V)$-Steiner system $(\mathcal{V},\mathcal{B})$ with embedding operators $\{\mathbf{E}_v\}_{v\in\mathcal{V}}$, we have
\begin{itemize}
\item[(a)]
$\sum_{v\in\mathcal{V}}\mathbf{E}_v\mathbf{E}_v^*=K \mathbf{I}_\mathcal{B}$,
\item[(b)]
$\mathbf{E}_v^*\mathbf{E}_v= \mathbf{I}_R$, and
\item[(c)]
for every $v,v'\in\mathcal{V}$ with $v\neq v'$, there exist $i,j\in[R]$ such that $\mathbf{E}_v^*\mathbf{E}_{v'}=\delta_i\delta_j^*$.
\end{itemize}
\end{lem}

\begin{proof}
For (a), we first note that, since $\mathbf{E}_v$ is an isometry, $\mathbf{E}_v\mathbf{E}_v^*$ is an orthogonal projection onto its range.
Since $\{\delta_b\}_{b\in\mathcal{B},~\! v\in b}$ is an orthonormal basis for this range, we have
\[
\mathbf{E}_v\mathbf{E}_v^*
=\sum_{\substack{b\in\mathcal{B}\\v\in b}}\delta_b\delta_b^*.
\]
The desired sum can then be simplified using this identity:
\[
\sum_{v\in\mathcal{V}}\mathbf{E}_v\mathbf{E}_v^*
=\sum_{v\in\mathcal{V}}\sum_{\substack{b\in\mathcal{B}\\v\in b}}\delta_b\delta_b^*
=\sum_{b\in\mathcal{B}}\sum_{\substack{v\in\mathcal{V}\\v\in b}}\delta_b\delta_b^*
=K\sum_{b\in\mathcal{B}}\delta_b\delta_b^*
=K \mathbf{I}_\mathcal{B}.
\]
Next, (b) is equivalent to $\mathbf{E}_v$ being an isometry.
For (c), suppose $v\neq v'$.
Then
\[
\langle \mathbf{E}_v^*\mathbf{E}_{v'}\delta_{r'},\delta_r\rangle
=\langle \mathbf{E}_{v'}\delta_{r'},\mathbf{E}_v\delta_r\rangle
=\langle \delta_{b'(r')},\delta_{b(r)}\rangle
\in\{0,1\}
\qquad
\forall r,r'\in[R],
\]
where $\delta_{b(r)}:=\mathbf{E}_v\delta_r$ and $\delta_{b'(r')}:=\mathbf{E}_v\delta_{r'}$.
Summing these quantities gives
\begin{align*}
\sum_{r=1}^R\sum_{r'=1}^R\langle \mathbf{E}_v^*\mathbf{E}_{v'}\delta_{r'},\delta_r\rangle
&=\sum_{r=1}^R\sum_{r'=1}^R\langle \delta_{b'(r')},\delta_{b(r)}\rangle\\
&=\sum_{\substack{b\in\mathcal{B}\\v\in b}}\sum_{\substack{b'\in\mathcal{B}\\v'\in b'}}\langle \delta_{b'},\delta_b\rangle\\
&=\sum_{\substack{b\in\mathcal{B}\\v\in b}}\sum_{\substack{b'\in\mathcal{B}\\v'\in b'}}\left\{\begin{array}{cl}1&\text{if }b=b'\\0&\text{else}\end{array}\right\}
=\#\{b\in\mathcal{B}:v,v'\in b\}
=1.
\end{align*}
We therefore conclude that $\langle \mathbf{E}_v^*\mathbf{E}_{v'}\delta_{r'},\delta_r\rangle=1$ for exactly one choice of $r$ and $r'$, and equals zero otherwise, thereby implying the result.
\end{proof}

These properties are useful in the proof of the following construction:

\begin{prop}[Steiner ETFs, see~\cite{FickusMT:12}]
\label{prop.steiner}
Given a $(2,K,V)$-Steiner system $(\mathcal{V},\mathcal{B})$ with embedding operators $\{\mathbf{E}_v\}_{v\in\mathcal{V}}$ and a unimodular simplex $\{\varphi_s\}_{s=1}^{R+1}$ in $\mathbb{C}^R$, the ensemble $\{\mathbf{E}_v\varphi_s\}_{s=1,~\! v\in \mathcal{V}}^{R+1}$ forms an equiangular tight frame in $\mathbb{C}^\mathcal{B}$.
\end{prop}

As an example, we use a $(2,3,7)$-Steiner system known as the Fano plane to construct a Steiner ETF of $28$ vectors in $\mathbb{C}^7$ (actually, $\mathbb{R}^7$).
The following matrix is organized into $V=7$ blocks, each of size $B\times(R+1)=7\times 4$, obtained by applying a corresponding embedding operator $\mathbf{E}_v$ to a fixed unimodular simplex extracted from a $4\times 4$ Hadamard matrix:
\[
\left[
\begin{smallmatrix}
+&-&+&-&+&-&+&-&+&-&+&-&&&&&&&&&&&&&&&&\\
+&+&-&-&&&&&&&&&+&-&+&-&+&-&+&-&&&&&&&&\\
+&-&-&+&&&&&&&&&&&&&&&&&+&-&+&-&+&-&+&-\\
&&&&+&+&-&-&&&&&+&+&-&-&&&&&+&+&-&-&&&&\\
&&&&+&-&-&+&&&&&&&&&+&+&-&-&&&&&+&+&-&-\\
&&&&&&&&+&+&-&-&+&-&-&+&&&&&&&&&+&-&-&+\\
&&&&&&&&+&-&-&+&&&&&+&-&-&+&+&-&-&+&&&&
\end{smallmatrix}
\right]
\]
(For convenience, we chose to scale the vectors so that each coordinate has size $0$ or $1$; the scaling has no effect on whether the columns form an ETF.)
Here, $+$ denotes $1$, $-$ denotes $-1$, and blank entries are zeros.
It is clear from this example that the columns have squared norm $3$ (the number of blocks incident to each point), columns from a common block have inner product $-1$ (since they belong to a common embedded simplex), and the supports of columns from different blocks intersect at a single entry (since two points determine a block), leading to an inner product of $\pm1$.
Additionally, tightness follows from the fact that the rows of this matrix are orthogonal with equal norms.

\begin{proof}[Proof of Proposition~\ref{prop.steiner}]
First, the unimodular simplex satisfies $\sum_{s=1}^{R+1}\varphi_s\varphi_s^*=(R+1)\mathbf{I}_R$, which gives tightness:
\[
\sum_{v\in\mathcal{V}}\sum_{s=1}^{R+1}(\mathbf{E}_v\varphi_s)(\mathbf{E}_v\varphi_s)^*
=\sum_{v\in\mathcal{V}}\mathbf{E}_v\bigg(\sum_{s=1}^{R+1}\varphi_s\varphi_s^*\bigg)\mathbf{E}_v^*
=(R+1)\sum_{v\in\mathcal{V}}\mathbf{E}_v\mathbf{E}_v^*
=K(R+1) \mathbf{I}_\mathcal{B},
\]
where the last step follows from Lemma~\ref{lem.upsampling}(a).
Next, Lemma~\ref{lem.upsampling}(b) gives
\[
\langle \mathbf{E}_v\varphi_s,\mathbf{E}_v\varphi_{s'}\rangle
=\langle \mathbf{E}_v^*\mathbf{E}_v\varphi_s,\varphi_{s'}\rangle
=\langle \varphi_s,\varphi_{s'}\rangle,
\]
which equals $\|\varphi_s\|^2=R$ when $s=s'$, and otherwise resides in the complex unit circle $\mathbb{T}$.
Finally, for $v\neq v'$, Lemma~\ref{lem.upsampling}(c) gives
\[
\langle \mathbf{E}_v\varphi_s,\mathbf{E}_{v'}\varphi_{s'}\rangle
=\langle \mathbf{E}_{v'}^*\mathbf{E}_v\varphi_s,\varphi_{s'}\rangle
=\langle \delta_i\delta_j^*\varphi_s,\varphi_{s'}\rangle
=\langle \delta_j^*\varphi_s,\delta_i^*\varphi_{s'}\rangle
=\varphi_s(j)\overline{\varphi_{s'}(i)}
\in\mathbb{T},
\]
thereby implying equiangularity.
\end{proof}

\section{Tremain ETFs}

What follows is our main result, which specifically uses $(2,3,V)$-Steiner systems, known as \textbf{Steiner triple systems}:

\begin{thm}[Tremain ETFs]
\label{thm.tremain}
Given a Steiner triple system $(\mathcal{V},\mathcal{B})$ with embedding operators $\{\mathbf{E}_v\}_{v\in\mathcal{V}}$, unimodular simplices $\{\varphi_s\}_{s=1}^{R+1}$ in $\mathbb{C}^R$ and $\{\psi_t\}_{t=1}^{V+1}$ in $\mathbb{C}^\mathcal{V}$, and corresponding Naimark complements $\{a_s\}_{s=1}^{R+1}$ and $\{b_t\}_{t=1}^{V+1}$ in $\mathbb{C}$, the ensemble
\[
\Big\{\mathbf{E}_v\varphi_s\oplus\sqrt{2}a_s\delta_v\oplus 0\Big\}_{s=1,~\! v\in \mathcal{V}}^{R+1}
\cup\Big\{0_\mathcal{B}\oplus\sqrt{\tfrac{1}{2}}\psi_t\oplus\sqrt{\tfrac{3}{2}}b_t\Big\}_{t=1}^{V+1}
\]
forms an equiangular tight frame in $\mathbb{C}^\mathcal{B}\oplus \mathbb{C}^\mathcal{V}\oplus \mathbb{C}$.
\end{thm}

For the sake of illustration, we turn to Example~7.10 in Tremain's notes~\cite{Tremain:09}.
Consider the columns of the following matrix:
\[
\left[
\begin{smallmatrix}
+&-&+&-&+&-&+&-&+&-&+&-&&&&&&&&&&&&&&&&&&&&&&&&\\
+&+&-&-&&&&&&&&&+&-&+&-&+&-&+&-&&&&&&&&&&&&&&&&\\
+&-&-&+&&&&&&&&&&&&&&&&&+&-&+&-&+&-&+&-&&&&&&&&\\
&&&&+&+&-&-&&&&&+&+&-&-&&&&&+&+&-&-&&&&&&&&&&&&\\
&&&&+&-&-&+&&&&&&&&&+&+&-&-&&&&&+&+&-&-&&&&&&&&\\
&&&&&&&&+&+&-&-&+&-&-&+&&&&&&&&&+&-&-&+&&&&&&&&\\
&&&&&&&&+&-&-&+&&&&&+&-&-&+&+&-&-&+&&&&&&&&&&&&\\
\bullet&\bullet&\bullet&\bullet&&&&&&&&&&&&&&&&&&&&&&&&&\blacksquare&\square&\blacksquare&\square&\blacksquare&\square&\blacksquare&\square\\
&&&&\bullet&\bullet&\bullet&\bullet&&&&&&&&&&&&&&&&&&&&&\blacksquare&\blacksquare&\square&\square&\blacksquare&\blacksquare&\square&\square\\
&&&&&&&&\bullet&\bullet&\bullet&\bullet&&&&&&&&&&&&&&&&&\blacksquare&\square&\square&\blacksquare&\blacksquare&\square&\square&\blacksquare\\
&&&&&&&&&&&&\bullet&\bullet&\bullet&\bullet&&&&&&&&&&&&&\blacksquare&\blacksquare&\blacksquare&\blacksquare&\square&\square&\square&\square\\
&&&&&&&&&&&&&&&&\bullet&\bullet&\bullet&\bullet&&&&&&&&&\blacksquare&\square&\blacksquare&\square&\square&\blacksquare&\square&\blacksquare\\
&&&&&&&&&&&&&&&&&&&&\bullet&\bullet&\bullet&\bullet&&&&&\blacksquare&\blacksquare&\square&\square&\square&\square&\blacksquare&\blacksquare\\
&&&&&&&&&&&&&&&&&&&&&&&&\bullet&\bullet&\bullet&\bullet&\blacksquare&\square&\square&\blacksquare&\square&\blacksquare&\blacksquare&\square\\
&&&&&&&&&&&&&&&&&&&&&&&&&&&&\Diamond&\Diamond&\Diamond&\Diamond&\Diamond&\Diamond&\Diamond&\Diamond
\end{smallmatrix}
\right]
\]
Here, $+$ and $-$ denote $\pm1$ and blanks are zeros as before, and furthermore
\[
\bullet=\sqrt{2},
\qquad
\blacksquare=\sqrt{\tfrac{1}{2}},
\qquad
\square=-\sqrt{\tfrac{1}{2}},
\qquad
\Diamond=\sqrt{\tfrac{3}{2}}.
\]
The first $28$ columns come from the Steiner ETF of the previous section.
Steiner ETF vectors in a common embedded simplex have inner product $-1$, and in the above matrix, these vectors recieve an additional entry in the $\mathbb{C}^\mathcal{V}$ component that changes their inner product to $+1$.
Otherwise, the Steiner-based vectors in the Tremain ETF do not interact in the $\mathbb{C}^\mathcal{V}$ or $\mathbb{C}$ components, thereby inheriting equiangularity from the Steiner ETF.
The other vectors are supported solely in the $\mathbb{C}^\mathcal{V}$ and $\mathbb{C}$ components: the portion in $\mathbb{C}^\mathcal{V}$ forms a simplex extracted from an $8\times 8$ Hadamard matrix whose inner products are $-1/2$, and including the $\mathbb{C}$ component changes these inner products to $+1$.
Also, the Steiner-based vectors have only one entry of support in common with these other vectors, allowing the inner products to have modulus $\sqrt{2}\cdot\sqrt{1/2}=1$.
Finally, the vectors all have squared norm $5$, corresponding to the Welch bound \eqref{eq.welch bound} for $M=16$ and $N=35$ (namely, $1/5$), and so this ensemble meets equality in the Welch bound, and equivalently forms an ETF. 

\begin{proof}[Proof of Theorem~\ref{thm.tremain}]
First, since $M=B+V+1$ and $N=V(R+1)+(V+1)$, a persistent application of the identities \eqref{eq.steiner identities} with $K=3$ simplifies the Welch bound in this case:
\[
\sqrt{\frac{N-M}{M(N-1)}}
=\frac{1}{R+2}.
\]
It suffices to demonstrate equality in the Welch bound, e.g., to show that each member of the purported ETF has squared norm $R+2$, and that inner products between distinct members lie in $\mathbb{T}$.
For convenience, we denote
\begin{equation}
\label{eq.names of tremain vectors}
\tau_{v,s}
:=\mathbf{E}_v\varphi_s\oplus\sqrt{2}a_s\delta_v\oplus 0
\qquad
\text{and}
\qquad
\tau_t
:=0_\mathcal{B}\oplus\sqrt{\tfrac{1}{2}}\psi_t\oplus\sqrt{\tfrac{3}{2}}b_t.
\end{equation}
Checking the norms is straightforward:
\begin{align*}
\|\tau_{v,s}\|^2
&=\|\mathbf{E}_v\varphi_s\|^2+\|\sqrt{2}a_s\delta_v\|^2
=\|\varphi_s\|^2+2|a_s|^2
=R+2,\\
\|\tau_t\|^2
&=\Big\|\sqrt{\tfrac{1}{2}}\psi_t\Big\|^2+\Big\|\sqrt{\tfrac{3}{2}}b_t\Big\|^2
=\tfrac{1}{2}\|\psi_t\|^2+\tfrac{3}{2}|b_t|^2
=\tfrac{1}{2}V+\tfrac{3}{2}
=R+2.
\end{align*}
To demonstrate equiangularity, we proceed in cases:

Case I: 
Suppose $v=v'$ but $s\neq s'$.
Then $\langle \mathbf{E}_v\varphi_s,\mathbf{E}_v\varphi_{s'}\rangle=\langle \mathbf{E}_v^*\mathbf{E}_v\varphi_s,\varphi_{s'}\rangle=\langle \varphi_s,\varphi_{s'}\rangle$ by Lemma~\ref{lem.upsampling}(b), and so the Naimark property \eqref{eq.naimark} gives
\begin{equation}
\label{eq.ip1}
\langle \tau_{v,s},\tau_{v',s'}\rangle
=\langle \mathbf{E}_v\varphi_s,\mathbf{E}_v\varphi_{s'}\rangle+2a_s\overline{a_{s'}}
=\langle \varphi_s,\varphi_{s'}\rangle+2a_s\overline{a_{s'}}
=a_s\overline{a_{s'}}.
\end{equation}
Case II:
Suppose $v\neq v'$.
Then for every $s,s'\in[R+1]$, Lemma~\ref{lem.upsampling}(c) gives
\begin{equation}
\label{eq.ip2}
\langle \tau_{v,s},\tau_{v',s'}\rangle
=\langle \mathbf{E}_v\varphi_s,\mathbf{E}_{v'}\varphi_{s'}\rangle
=\langle \mathbf{E}_{v'}^*\mathbf{E}_v\varphi_s,\varphi_{s'}\rangle
=\langle \delta_i\delta_j^*\varphi_s,\varphi_{s'}\rangle
=\varphi_s(j)\overline{\varphi_{s'}(i)}.
\end{equation}
Case III:
Suppose $t\neq t'$.
Then the Naimark property \eqref{eq.naimark} gives
\begin{equation}
\label{eq.ip3}
\langle\tau_t,\tau_{t'}\rangle
=\tfrac{1}{2}\langle\psi_t,\psi_{t'}\rangle+\tfrac{3}{2}b_t\overline{b_{t'}}
=\tfrac{1}{2}\Big(\langle\psi_t,\psi_{t'}\rangle+b_t\overline{b_{t'}}\Big)+b_t\overline{b_{t'}}
=b_t\overline{b_{t'}}.
\end{equation}
Case IV:
For every $v\in\mathcal{V}$, $s\in[R+1]$ and $t\in[V+1]$, we have
\begin{equation}
\label{eq.ip4}
\langle\tau_{v,s},\tau_t\rangle
=\Big\langle \sqrt{2}a_s\delta_v,\sqrt{\tfrac{1}{2}}\psi_t\Big\rangle
=a_s\psi_t(v).
\end{equation}
All of these inner products lie in $\mathbb{T}$, and so we are done.
\end{proof}

Since Steiner triple systems exist if and only if $V\equiv 1$ or $3\bmod 6$ with $V\geq3$, and since complex unimodular simplices exist in every dimension, we have the following corollary:

\begin{cor}
For every $V\equiv 1$ or $3\bmod 6$ with $V\geq3$, there exists an equiangular tight frame of $N$ vectors in $\mathbb{C}^M$ with
\[
M=\tfrac{1}{6}(V+2)(V+3),
\qquad
N=\tfrac{1}{2}(V+1)(V+2).
\]
\end{cor}

As illustrated by Tremain's original example above, this construction is real whenever the unimodular simplices come from real Hadamard matrices.
In particular, a $(V+1)\times(V+1)$ real Hadamard matrix may be obtained from a Kronecker product of a $(R+1)\times(R+1)$ real Hadamard matrix with a $2\times 2$ real Hadamard matrix.
This gives another corollary:

\begin{cor}
\label{cor.real tremain}
If there exists an $h\times h$ real Hadamard matrix with $h\equiv 1$ or $2\bmod 3$, then there exists an equiangular tight frame of $N$ vectors in $\mathbb{R}^M$ with
\[
M=\tfrac{1}{3}(h+1)(2h+1),
\qquad
N=h(2h+1).
\]
\end{cor}

The following section demonstrates the novelty of these ETFs by using them to construct new strongly regular graphs and distance-regular antipodal covers of the complete graph.

\section{New SRGs and DRACKNs}

Recall that a $(v,k,\lambda,\mu)$-strongly regular graph (SRG) is a $v$-vertex graph such that every vertex has the same number $k$ of neighbors, that adjacent vertices have the same number $\lambda$ of common neighbors, and that non-adjacent vertices also have the same number $\mu$ of common neighbors.
There are two ways that certain ETFs can be associated with SRGs.
First, there is a one-to-one correspondence between real ETFs (with $N>M$) and a certain class of SRGs, specifically SRGs with parameters satisfying $2\mu=k$:

\begin{prop}[see~\cite{Waldron:09,FickusW:15}]
\label{Waldron}
Suppose there exists an equiangular tight frame of $N$ vectors in $\mathbb{R}^M$ with $N>M$, and set
\[
\beta = \sqrt{\frac{N-M}{M(N-1)}}.
\]
Then there exists a strongly regular graph with parameters
\[
v=N-1,
\quad 
k=\frac{N}{2} - 1 + \frac{1}{2\beta}\left(\frac{N}{M}-2\right),
\quad 
\lambda = \frac{3k-v-1}{2},
\quad 
\mu=\frac{k}{2}.
\]
\end{prop}

Combined with Corollary~\ref{cor.real tremain}, this then implies the following:

\begin{cor}
\label{cor.real tremain srg}
If there exists an $h\times h$ real Hadamard matrix with $h\equiv 1$ or $2\bmod 3$, then there exists a strongly regular graph with parameters
\[
v=(2h-1)(h+1),
\quad
k = (h+2)(h-1),
\quad
\lambda = \tfrac{1}{2} (h^2+2h-6),
\quad
\mu=\tfrac{1}{2}(h+2)(h-1).
\]
\end{cor}

Examining small values of $h\equiv 1$ or $2\bmod 3$ for which there exists a real Hadamard matrix, we list the corresponding real Tremain ETFs as well as the parameters of the associated SRGs given by Corollary~\ref{cor.real tremain srg}:

\begin{center}
\begin{tabular}{c|c|c|c|c|c|c|c}
$h$ & $M$ & $N$ & $v$ & $k$ & $\lambda$ & $\mu$ & New?\\
\hline
$2$ & $5$ & $10$ & $9$ & $4$ & $1$ & $2$ & No\\
$4$ & $15$ & $36$ & $35$ & $18$ & $9$ & $9$ & No\\
$8$ & $51$ & $136$ & $135$ & $70$ & $37$ & $35$ & No\\
$16$ & $187$ & $528$ & $527$ & $270$ & $141$ & $135$ & No\\
$20$ & $287$ & $820$ & $819$ & $418$ & $217$ & $209$ & Yes\\
$28$ & $551$ & $1596$ & $1595$ & $810$ & $417$ & $405$ & ?\\
\end{tabular}
\end{center}

Examining Brouwer's table of strongly regular graphs~\cite{Brouwer:online,Brouwer:96}, we see that the SRG with parameters $v=819$, $k=418$, $\lambda=217$ and $\mu=209$ appears to be new to the literature.
The SRG with parameters $v=1595$, $k=810$, $\lambda = 417$ and $\mu=405$ is too large to be reported on the table.

Of course, the Hadamard conjecture implies that there is an infinite family of real Tremain ETFs, but this is unnecessary.
Indeed, Kronecker powers of the $2\times 2$ Hadamard matrix easily produce an infinite family of admissible Hadamard matrices.
In addition, it is not difficult to verify using Dirichlet's theorem on primes in arithmetic progressions that Paley's construction is admissible infinitely often.
Moreover, admissible Hadamard matrices are closed under Kronecker products, so there is actually a large family of explicit Hadamard matrices available to the real Tremain ETF construction.

While all real ETFs give rise to an SRG by way of Proposition~\ref{Waldron}, there is another method of constructing different SRGs from certain ETFs.
A special case of this construction of SRGs was first observed by Goethals and Seidel~\cite{GoethalsS:70}.

\begin{prop}[see~\cite{FickusJMPW:15}]
\label{GS}
Suppose there exists an equiangular tight frame $\{\varphi_i\}_{i=1}^N$ in $\mathbb{R}^M$ with $N>M$, and set
\[
\alpha = \frac{N}{M},
\qquad 
\beta = \sqrt{\frac{N-M}{M(N-1)}}.
\]
If there exists $x\in\mathbb{R}^M$ such that $\langle x,\varphi_i\rangle=1$ for every $i\in[N]$, then there exists a strongly regular graph with parameters
\[
v=N,
\quad
k=\frac{N-1}{2}+\frac{\alpha-1}{2\beta},
\quad
\lambda = \frac{N}{4} - 1 + \frac{3\alpha-4}{4\beta},
\quad
\mu = \frac{N}{4}+\frac{\alpha}{4\beta}.
\]
\end{prop}

Note that the additional hypothesis in Proposition~\ref{GS} asks for a vector $x\in\mathbb{R}^M$ such that $\langle x,\varphi_i\rangle=1$ for every $i\in[N]$.
We will show that real Tremain ETFs can be constructed to satisfy this property if the underlying Steiner triple system contains a \textbf{parallel class}, that is, a subset of the blocks that form a partition of the set of points.
A parallel class of a $(2,K,V)$-Steiner system must contain exactly $V/K$ blocks, and so the existence of a parallel class requires $K$ to divide $V$.
Fortunately, it is known that there exists a Steiner triple system with a parallel class whenever $V\equiv3\bmod6$~\cite{ColbournM:07}.
This leads to the following:

\begin{cor}
\label{TGScor}
If there exists an $h\times h$ real Hadamard matrix with $h\equiv 2\bmod 3$, then there exists a strongly regular graph with parameters
\[
v=h(2h+1),
\quad 
k = \frac{(2h-1)(h+2)}{2},
\quad 
\lambda = \frac{(h-1)(h+4)}{2},
\quad 
\mu = \frac{h(h+2)}{2}.
\]
\end{cor}

\begin{proof}
Consider real Hadamard matrices $\mathbf{H}_1$ and $\mathbf{H}_2$ of sizes $h\times h$ and $2h\times 2h$, respectively.
Note that we may select $\mathbf{H}_1$ and $\mathbf{H}_2$ so that each has all $1$s in the first row.
Put $R=h-1$ and $V=2h-1$, let $\{\varphi_s\}_{s=1}^{R+1}$ denote the columns of the matrix $\mathbf{\Phi}$ obtained by removing the last row $a$ of $\mathbf{H}_1$, and let $\{\psi_t\}_{t=1}^{V+1}$ denote the columns of the matrix $\mathbf{\Psi}$ obtained by removing the first row $b$ of $\mathbf{H}_2$.
Since $h\equiv2\bmod3$, we have $V\equiv3\bmod6$, and so we may select a Steiner triple system with a parallel class $\mathcal{P}\subseteq\mathcal{B}$.
For each $v\in\mathcal{V}$, pick an embedding operator $\mathbf{E}_v$ such that $\mathbf{E}_v\delta_1=\delta_b$ for the unique $b\in\mathcal{P}$ such that $v\in b$.
Combine these ingredients to produce a real Tremain ETF according to Theorem~\ref{thm.tremain}.

Now pick $x\in\mathbb{C}^\mathcal{B}\oplus\mathbb{C}^\mathcal{V}\oplus\mathbb{C}$ of the form $x=\chi\oplus 0\oplus\sqrt{2/3}$, where $\chi(b)=1$ for every $b\in\mathcal{P}$, and equals zero otherwise.
Then for every $v\in\mathcal{V}$, our choice of $\mathbf{E}_v$ ensures that $\mathbf{E}_v^*\chi=\delta_1$.
Appealing to the notation in \eqref{eq.names of tremain vectors}, we therefore have
\[
\langle x,\tau_{v,s}\rangle
=\langle \chi,\mathbf{E}_v\varphi_s\rangle
=\langle \mathbf{E}_v^*\chi,\varphi_s\rangle
=\langle \delta_1,\varphi_s\rangle
=\overline{\varphi_s(1)}
=1.
\]
Also, $\langle x,\tau_t\rangle=\sqrt{\tfrac{2}{3}}\cdot\sqrt{\tfrac{3}{2}}\overline{b_t}=1$, and so the result follows from Proposition~\ref{GS}.
\end{proof}

For small values of $h$, we have the following table of SRG parameters that arise via Corollary~\ref{TGScor}:

\begin{center}
\begin{tabular}{c|c|c|c|c|c|c|c}
$h$ & $M$ & $N$ & $v$ & $k$ & $\lambda$ & $\mu$ & New?\\
\hline
$2$ & $5$ & $10$ & $10$ & $6$ & $3$ & $4$ & No\\
 $8$ & $   77    $ & $  136    $ & $  136    $ & $   75    $ & $   42    $ & $   40$ & No\\
$20$ & $  431    $ & $  820    $ & $  820    $ & $  429    $ & $  228    $ & $  220$ & Yes\\
$32$ & $ 1073    $ & $ 2080    $ & $ 2080    $ & $ 1071    $ & $  558    $ & $  544$ & ?\\
$44$ & $ 2003    $ & $ 3916    $ & $ 3916    $ & $ 2001    $ & $ 1032    $ & $ 1012$ & ?\\
$56$ & $ 3221    $ & $ 6328    $ & $ 6328    $ & $ 3219    $ & $ 1650    $ & $ 1624$ & ?
\end{tabular}
\end{center}

Again, we consult Brouwer's table of strongly regular graphs~\cite{Brouwer:online,Brouwer:96} to identify which SRGs are known.
The SRG with parameters $v=820$, $k=429$, $\lambda=228$ and $\mu=220$ seems to be new, and the last three graphs from the table above are too large to be reported on the table. 

Having constructed a couple of new families of SRGs, we turn to another interesting class of graphs.
An \textbf{$(n,r,c)$-distance-regular antipodal cover of the complete graph} (DRACKN) is a simple graph $G$ on $rn$ vertices along with a partition of the vertices into $n$ sets (called \textbf{fibers}) of $r$ vertices such that
\begin{itemize}
\item[(i)]
no two vertices in the same fiber are adjacent,
\item[(ii)]
the edge set contains a perfect matching between any two fibers, and
\item[(iii)]
non-adjacent vertices have $c$ common neighbors.
\end{itemize}
DRACKNs have been investigated in~\cite{GodsilH:92,KlinP:11,CountinhoGSZ:15}.
Recently, Coutinho, Godsil, Shirazi and Zhan~\cite{CountinhoGSZ:15} proved the following result that produces DRACKNs from certain ETFs:

\begin{prop}[Theorem~5.1 in~\cite{CountinhoGSZ:15}]
\label{DRACKN}
Let $p$ be prime.
If there exists an equiangular tight frame $\{\varphi_i\}_{i=1}^N$ in $\mathbb{C}^M$ such that $\langle \varphi_i,\varphi_j\rangle$ is a $p$th root of unity whenever $i\neq j$, then there exists an $(N,p,c)$-distance-regular antipodal cover of the complete graph, where
\[
c = \frac{1}{p}\left(N-2+\frac{2M-N}{\beta M}\right),
\quad
\quad \beta = \sqrt{\frac{N-M}{M(N-1)}}.\]
\end{prop}

In what follows, we show how Tremain ETFs can be constructed to satisfy the hypothesis of Proposition~\ref{DRACKN} and hence produce DRACKNs.
The crucial ingredient here is the existence of so-called \textbf{Butson-type Hadamard matrices}, that is, complex Hadamard matrices whose entries are roots of unity.
Notationally, an $H(p,h)$ matrix is any $h\times h$ Hadamard matrix whose entries are all $p$th roots of unity.

\begin{cor}
\label{cor.tremain drackn}
Let $p$ be prime.
If there exist $H(p,h)$ and $H(p,2h)$ matrices for some $h\equiv 1$ or $2\bmod 3$, then there exists a $(h(2h+1),p,2h^{2}/p)$-distance-regular antipodal cover of the complete graph.
\end{cor}

\begin{proof}
Consider $H(p,h)$ and $H(p,2h)$ matrices $\mathbf{H}_1$ and $\mathbf{H}_2$, respectively.
Put $R=h-1$ and $V=2h-1$, let $\{\varphi_s\}_{s=1}^{R+1}$ denote the columns of the matrix $\mathbf{\Phi}$ obtained by removing the first row $a$ of $\mathbf{H}_1$, and let $\{\psi_t\}_{t=1}^{V+1}$ denote the columns of the matrix $\mathbf{\Psi}$ obtained by removing the first row $b$ of $\mathbf{H}_2$.
Since $h\equiv 1$ or $2\bmod 3$, we have $V\equiv 1$ or $3 \bmod6$, and so we may select a Steiner triple system with appropriate embedding operators, and combine these ingredients to produce a Tremain ETF according to Theorem~\ref{thm.tremain}.

By Proposition~\ref{DRACKN}, it suffices to show that the inner products between all pairs of vectors in this ETF are $p$th roots of unity.
These inner products are expressed in equations \eqref{eq.ip1}--\eqref{eq.ip4}.
Since $a_s$, $b_t$, $\varphi_s(i)$ and $\psi_t(v)$ are $p$th roots of unity for every $s\in[R+1]$, $t\in[V+1]$, $i\in[R]$ and $v\in\mathcal{V}$, we are done.
\end{proof}

It is conjectured that there exists an $H(p,\lambda p)$ matrix for every prime $p$ and every positive integer $\lambda$.
One result in this direction is the following, due to Butson himself:

\begin{prop}[see~\cite{Butson:62}]
If $p$ is prime and $k\geq m\geq0$ are integers, then there exists an $H(p,2^{m}p^{k})$ matrix.
\end{prop}

As such, the construction in Corollary~\ref{cor.tremain drackn} gives an infinite family of DRACKNs.
To demonstrate the novelty of this construction, we first note that by Theorem~5.1 in~\cite{CountinhoGSZ:15}, our DRACKNs are abelian in a certain (technical) sense.
Existing abelian DRACKNs are enumerated in Table~1 in~\cite{CountinhoGSZ:15}, and all of these either have $n=rc$ or $(n,r,c)=(45,3,12)$.
By contrast, the DRACKNs from Corollary~\ref{cor.tremain drackn} satisfy $n-rc=h\neq0$.


\begin{thebibliography}{WW}

\bibitem{Brouwer:online}
A.\ E.\ Brouwer,
Parameters of strongly regular graphs,
Available online: \url{https://www.win.tue.nl/~aeb/graphs/srg/srgtab.html}

\bibitem{Brouwer:96}
A.\ E.\ Brouwer,
Strongly regular graphs,
in:\ The CRC Handbook of Combinatorial Designs, C.\ J.\ Colbourn, J.\ H.\ Dinitz, eds., CRC Press, 2007, pp 667--685.

\bibitem{Butson:62}
A.\ T.\ Butson,
Generalized Hadamard matrices,
Proc.\ Am.\ Math.\ Soc.\ 13 (1962) 894--898.

\bibitem{ColbournM:07}
C.\ J.\ Colbourn, R.\ Mathon,
Steiner systems,
in:\ The CRC Handbook of Combinatorial Designs, C.\ J.\ Colbourn, J.\ H.\ Dinitz, eds., CRC Press, 2007, pp 667--685.

\bibitem{CountinhoGSZ:15}
G.\ Coutinho, C.\ Godsil, M.\ Shirazi, H.\ Zhan,
Equiangular lines and covers of the complete graph,
Linear Algebra Appl.\ 488 (2016) 264--283.

\bibitem{FickusJMPW:15}
M.\ Fickus, J.\ Jasper, D.\ G.\ Mixon, J.\ D.\ Peterson, C.\ E.\ Watson,
Equiangular tight frames with centroidal symmetry,
Available online: arXiv:1509.04059

\bibitem{FickusM:15}
M.\ Fickus, D.\ G.\ Mixon,
Tables of the existence of equiangular tight frames,
Available online: arXiv:1504.00253

\bibitem{FickusMT:12}
M.\ Fickus, D.\ G.\ Mixon, J.\ C.\ Tremain,
Steiner equiangular tight frames,
Linear Algebra Appl.\ 436 (2015) 1014--1027.

\bibitem{FickusW:15}
M.\ Fickus, C.\ E.\ Watson,
Detailing the equivalence between real equiangular tight frames and certain strongly regular graphs,
Proc.\ SPIE 9597 (2015) 959719/1--10.

\bibitem{GodsilH:92}
C.\ D.\ Godsil, A.\ D.\ Hensel,
Distance regular covers of the complete graph,
J.\ Combin.\ Theory Ser.\ B 56 (1992) 205--238.

\bibitem{GoethalsS:70}
J.\ M.\ Goethals, J.\ J.\ Seidel,
Strongly regular graphs derived from combinatorial designs,
Can.\ J.\ Math.\ 22 (1970) 597--614.

\bibitem{KlinP:11}
M.\ Klin, C.\ Pech,
A new construction of antipodal distance regular covers of complete graphs through the use of Godsil--Hensel matrices,
Ars Math.\ Contemp.\ 4 (2011) 205--243.

\bibitem{Rankin:56}
R.\ A.\ Rankin,
On the minimal points of positive definite quadratic forms,
Mathematika 3 (1956) 15--24.

\bibitem{StrohmerH:03}
T.\ Strohmer, R.\ W.\ Heath,
Grassmannian frames with applications to coding and communication,
Appl.\ Comput.\ Harmon.\ Anal.\ 14 (2003) 257--275.

\bibitem{Tremain:09}
J.\ C.\ Tremain,
Concrete constructions of real equiangular line sets,
Available online, 2009.

\bibitem{vanLintS:66}
J.\ H.\ van Lint, J.\ J.\ Seidel,
Equilateral point sets in elliptic geometry,
Indag.\ Math.\ 28 (1966) 335--348.

\bibitem{Waldron:09}
S.\ Waldron,
On the construction of equiangular frames from graphs,
Linear Algebra Appl.\ 431 (2009) 2228--2242.

\bibitem{Welch:74}
L.\ R.\ Welch,
Lower bounds on the maximum cross correlation of signals,
IEEE Trans.\ Inf.\ Theory 20 (1974) 397--399.






\end{thebibliography}
\end{document}